\documentclass[12pt]{amsart}
\usepackage{graphicx}
\usepackage{graphics}
\usepackage{amsmath}
\usepackage{amscd}
\usepackage{latexsym}
\usepackage{amssymb}
\usepackage{amsthm}

\usepackage[usenames, dvipsnames]{color}

\usepackage{blindtext}
\usepackage{scrextend}
\addtokomafont{labelinglabel}{\sffamily}

\usepackage{hyperref}  
\hypersetup{
  colorlinks   = true, 
  urlcolor     = blue, 
  linkcolor    = blue, 
  citecolor    = blue  
}

\newcommand{\hg}[1]{\text{H}_{#1}}

\begin{document}

\textwidth 6.2in
\textheight 7.6in
\evensidemargin .75in
\oddsidemargin.75in

\newtheorem{Thm}{Theorem}
\newtheorem{Lem}[Thm]{Lemma}
\newtheorem{Cor}[Thm]{Corollary}
\newtheorem{Prop}[Thm]{Proposition}
\newtheorem{Rm}{Remark}
\newtheorem{Qu}{Question}
\newtheorem{Def}{Definition}
\newtheorem{Exm}{Example}

\def\a{{\mathbb a}}
\def\C{{\mathbb C}}
\def\A{{\mathbb A}}
\def\B{{\mathbb B}}
\def\D{{\mathbb D}}
\def\E{{\mathbb E}}
\def\R{{\mathbb R}}
\def\P{{\mathbb P}}
\def\S{{\mathbb S}}
\def\Z{{\mathbb Z}}
\def\O{{\mathbb O}}
\def\H{{\mathbb H}}
\def\V{{\mathbb V}}
\def\Q{{\mathbb Q}}
\def\Cn{${\mathcal C}_n$}
\def\CM{\mathcal M}
\def\CG{\mathcal G}
\def\CH{\mathcal H}
\def\CT{\mathcal T}
\def\CF{\mathcal F}
\def\CA{\mathcal A}
\def\CB{\mathcal B}
\def\CD{\mathcal D}
\def\CP{\mathcal P}
\def\CS{\mathcal S}
\def\CZ{\mathcal Z}
\def\CE{\mathcal E}
\def\CL{\mathcal L}
\def\CV{\mathcal V}
\def\CW{\mathcal W}
\def\IC{\mathbb C}
\def\IF{\mathbb F}
\def\IK{\mathcal K}
\def\IL{\mathcal L}
\def\IP{\bf P}
\def\IR{\mathbb R}
\def\IZ{\mathbb Z}

\title{A Note on Knot Concordance}
\author{EYLEM ZEL\.{I}HA YILDIZ}
\address{Department of Mathematics, Michigan State University, MI, 48824}
\email{ezyildiz@msu.edu}
\date{\today}

\begin{abstract} 
We use classical techniques to answer some questions raised by Daniele Celoria about almost-concordance of knots in arbitrary closed $3$-manifolds. We first prove that, given $Y^3 \neq S^3$, for any non-trivial element $g\in \pi_1(Y)$ there are infinitely many distinct smooth almost-concordance classes in the free homotopy class of the unknot. In particular we consider these distinct smooth almost-concordance classes on the boundary of a Mazur manifold and we show none of these distinct classes bounds a PL-disk in the Mazur manifold, but all the representatives we construct are topologically slice. We also prove that all knots in the free homotopy class of $S^1 \times pt$ in $S^1 \times S^2$ are smoothly concordant.
\end{abstract}

\date{}
\maketitle

\setcounter{section}{-1}

\vspace{-.35in}

\section{Introduction}

\vspace{.1in}

In this work we consider manifolds that are smooth and oriented. Let $Y$ be a closed, connected, oriented $3$-manifold. A \emph{knot} $k$ in $Y$ is an isotopy class of a smooth embeddings $S^1 \hookrightarrow Y$. Two knots $k_1$ and $k_2$ are said to be \emph{concordant} if there is a smooth proper embedding of an annulus $F: S^1\times [0,1] \hookrightarrow Y\times[0,1] $, such that its boundary $\partial{F(S^1\times [0,1])} = k_1\times \{0\} \cup (-k_2)  \times \{1\}$ where $(-k_2)$ is the same knot $k_2$ with the reversed orientation. If we allow $F$ to have only finitely many singular points, all of which are cones over knots, then $k_1$ and $k_2$ are called \emph{PL-concordant}. We call these knots \emph{singular concordant} if we allow $F$ to be an immersion instead an embedding. Two knots are singular concordant if and only if they are freely homotopic. One can see this fact by using the Immersion Theorems and general position arguments (can be found in \cite{hm1}) on the trace of homotopy.
 
\noindent Concordance is an equivalence relation $\sim $ on the set of oriented knots in $Y$. The set of equivalence classes is denoted by;
$$ \mathcal{C}(Y) =\{ \text{Oriented knots in Y}\}/\sim .$$
Concordant knots $k_1$ and $k_2$ are freely homotopic, hence they are homologous. 
In \cite{c1} Daniele Celoria defines the concept of almost-concordance of 
knots. Two knots $k_1$ and $k_2$ in $Y$ are said to be \emph{almost-concordant} if there are $ k_1^{'},\; k_2^{'} \subset S^3 $ such that $ k_1 \# k_1^{'} \sim k_2 \# k_2^{'} $, and this is expressed by $ k_1 \dot{\sim} k_2 $. Like concordance, almost-concordance is an equivalence relation, and it implies free homotopy of knots.\newline We denote almost-concordance classes by $\widetilde{\mathcal{C}}(Y)$. More generally;
$$\mathcal{C}_{\gamma}(Y):= \mathcal{K}_{\gamma}(Y)/\sim, $$ 
$$\widetilde{\mathcal{C}}_{\gamma}(Y):= \mathcal{K}_{\gamma}(Y)/\dot{\sim} \;\; ,$$
 where  $\mathcal{K}_{\gamma}(Y)$ is the free homotopy class of a knot $\gamma $ in $Y$.

\begin{Thm} \label{main1}
Given a closed $3$-manifold $Y$, for a non trivial element $g \in \pi_1(Y)$ we can construct infinitely many distinct almost-concordance classes in the free homotopy class of the unknot. If $h\not\in\{g,g^{-1}\}$, then the almost-concordance classes constructed (as in Figure~\ref{figure3}) using $g$ and $h$ are disjoint.
\end{Thm}

\noindent A question raised in \cite{c1} is \lq\lq Does there exists a pair $(Y,m)$ such that $C_m(Y)$ is finite?"  Theorem~\ref{main2} provide a positive answer.

\begin{Thm}\label{main2}
All knots in the free homotopy class of $S^1 \times pt$ in $S^1 \times S^2$ are smoothly concordant, i.e., $|\mathcal{C}_x(S^1\times S^2)|=1$  where $x$ represents $S^1\times pt $ in $S^1 \times S^2$. 
\end{Thm}

\noindent After this paper was posted a similar result to Theorem~\ref{main2} is appeared in \cite{dnpr}. In \cite{fnop} there are also related results to the above Theorems in the topological category. 

\vspace{.2cm}
\noindent\textbf{Acknowledgement.} I would like to thank Matthew Hedden for encouraging me to think about this problem and for valuable suggestions, and also I thank Selman Akbulut for helpful discussions.

\section{Wall's Self Intersection Number, a Concordance Invariant for Null-homotopic Knots, and Proof of Theorem \ref{main1}}

There are many approaches to knot concordance problem; here we focus on one of the classical techniques. This technique is based on Wall's intersection number \cite{w1}. The application of this idea to knot concordance was studied in \cite{s1} by Schneiderman. 

\noindent Let $k$ be a null-homotopic knot in $Y$; consider a singular concordance of $k$ to the unknot $u$, after capping with a disk the unknot, we get a proper immersion of a disk $ D \looparrowright Y \times I$ with $k= \partial D$. Let $p$ be a transverse self intersection of the immersion $D$ then any small neighbourhood of $p$ looks like two surfaces intersecting at $p$. These surfaces are called \emph{sheets}. The self-intersection number of $k$, defined as Wall's self intersection number of $D$, takes its value in the group ring $\Z[\pi_1 Y]$. To define this self intersection number we first fix a path from the base point $y_0$ of $Y \times I$ to a basepoint of the immersed disk $D$, called a \emph{whisker} of $D$. Now $g_p \in \pi_1{(Y,y_0)}$ is defined in the following way: it is a loop starting from $y_0$ going to the basepoint of $D$ using the whisker, then to the self-intersection point $p$ of $D$, then changing the sheet at the intersection point going back to the basepoint of $D$, and finally to $y_0$ using the whisker.
$$ \mu(k) :=\mu(D) = \sum_{p} sign(p)\cdot g_p \in \Z[\pi_1 Y].$$ 
\noindent Since $D$ is simply connected, the loop $g_p$ does not depend on the path we choose while travelling on $D$ as long as it stays away from self-intersection points. $sign(p)$ is $+1$ if the orientation of $Y\times I$ at $p$ matches with the orientation induced from sheets of $D$ at $p$, and it is $-1$ otherwise. After fixing the whisker there is still an indeterminacy coming from the choice of the first sheet. Altering this choice changes the loop from $g_p$ to $g_p^{-1}$.  Also, self-intersection points coming from cusp homotopies give elements which are trivial in $\pi_1(Y)$. Since we are interested in a homotopy invariant, we also quotient out these elements, arriving at the following abelian group   $$ \widetilde{\Lambda}:= \frac{ \mathbb{Z}[\pi_1Y ] }{ \{ g-g^{-1} \;|\; g \in \pi_1(Y)  \} \oplus \mathbb{Z}[1]} . $$

\noindent Here $\mathbb{Z}[1]$ is the abelian subgroup generated by the trivial element of $\pi_1(Y)$. Homotopy invariance in the above discussion follows from the following two Propositions.

\begin{Prop}[from Chapter 1.6 of \cite{fq1}] A homotopy between immersions of a surface in a $4-$manifold is homotopic to a composition of homotopies, each of which is a regular homotopy or a cusp homotopy in some ball, or the inverse of a cusp homotopy.

\end{Prop}

\begin{Prop}[from Chapter 1.7 of \cite{fq1}]
Intersection numbers and reduced self intersection numbers in $ \widetilde{\Lambda}$ are invariant under homotopy rel boundary. The $\Z[1]$ component of the self intersection number is invariant under regular homotopy, and conversely two immersions of a sphere or disk which are homotopic rel boundary, and have the same framed boundary, are regularly homotopic rel boundary if and only if the $\Z[1]$ component of the self intersection numbers are equal. 
\end{Prop}

\noindent Now we state and prove Schneiderman's knot concordance invariant.

\begin{Thm}[\cite{s1}]\label{null} The map   $$ \begin{array}{cccc}
\mu: & \mathcal{C}_1(Y) & \to & \widetilde{\Lambda} \\
 & k &\mapsto & \mu(k)   
\end{array} $$ is well defined and onto.
\end{Thm}

\begin{proof} We recall the proof from \cite{s1}.

\noindent \textbf{Well Defined:} Let $D$ and $D^{'}$ be singular null-concordances of a knot $k$, taking a singular sphere $ S=D \underset{k}{\cup} D^{'} \subset Y\times I $  gives $ S\in\pi_2(Y\times I) = \pi_2(Y)$. By \cite{h1} Proposition 3.12, there exists a disjoint collection of embedded 2-spheres generating $\pi_2(Y)$ as a $\pi_1(Y)-$module. Tubing these generators together in $Y\times I $ we get an embedded sphere in $Y \times I$. This implies $$\mu(S) = 0 = \mu(D) - \mu(D^{'}),$$ 
therefore $\mu(k)$ doesn't depend on $D$. 

\vspace{.05in}
\noindent \textbf{Concordance Invariance:} If $k_1, k_2 \in \mathcal{C}_1(Y)$ and $k_1\sim k_2$ then $\mu(k_2) = \mu(C\cup D) = \mu(D) = \mu(k_1) $ where $C$ is a concordance from $k_1$ to $k_2$, and $D$ is the singular concordance of $k_2$.
 
 \vspace{.05in}
\noindent \textbf{Surjectivity:} To construct $\pm g\in \Z[\pi_1(Y)]$ start with  an unknot $u$ and push an arc from $u$ around a loop representing $g\in\pi_1(Y)$ and create a $\pm $ clasp as in Figure~\ref{figure1}. Iterating this process one can get any desired element in $ \Z[\pi_1(Y)]$ via connected summing of such knots. \end{proof}
\begin{figure}[ht]
\begin{center}  
\includegraphics[scale=.8]{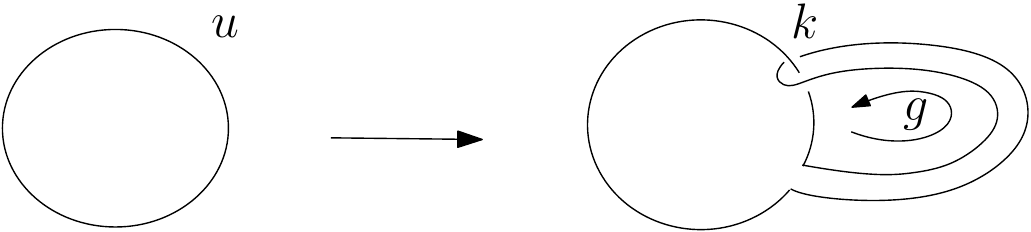}   
\caption{$\mu(k) = g$} 
\label{figure1} 
\end{center}
\end{figure} 

\begin{Lem}
 \label{lemma6} For any knots $k \in \mathcal{K}_1(Y)$,  $k^{'} \subset S^3$ we have  \begin{equation*}
 \mu (k \# k^{'} ) = \mu(k).
 \end{equation*} 
 This implies that  $\mu: \widetilde{\mathcal{C}}_1(Y) \to \widetilde{\Lambda}$ is well defined and onto. 
\end{Lem}

\begin{proof}
 We will construct a singular disk which will give us the desired result. By definition, $k$ bounds a proper immersion of a disk $ D \subset Y\times I $, and similarly $k^{'}$ bounds $D^{'} \subset S^3\times I$. Any band sum $D \underset{b}{\#}D^{'}$ where the interior of $b$ is away from $k$ and $k^{'}$  gives a proper immersion of a disk in $Y\times I$ bounded by $k\#k^{'}$. Take the base point and the whisker of $D$ as a base point and a whisker for $D \underset{b}{\#}D^{'}$ so
$$ \mu(D \underset{b}{\#}D^{'}) = \mu(D) + \beta \mu(D^{'}) \beta^{-1},$$ 
where $\beta \in \pi_1(Y)$ is determined by the band $b$ and the whisker. On the other hand $\pi_1(S^3)=1$ and $D^{'}$ lies entirely in $S^3 \times I$ therefore $\beta \mu(D^{'}) \beta^{-1} = 0 \in \widetilde{\Lambda}$ hence
$$  \mu(D \underset{b}{\#}D^{'}) = \mu(D)\;,\; \text{and}\;\;  \mu(k \underset{b}{\#}k^{'}) = \mu(k). \qedhere $$ \end{proof} 

\noindent This observation implies that Schneiderman's concordance invariant $\mu$ is also an almost-concordance invariant on freely null-homotopic knots. 

\begin{proof}[Proof of Theorem \ref{main1}]

By Theorem~\ref{null} and Lemma~\ref{lemma6}, $\mu: \widetilde{\mathcal{C}}_1(Y) \to \widetilde{\Lambda}$ is well defined, onto, and is an almost-concordance invariant on null-homotopic knots. For every non-trivial element $g\in\pi_1(Y)$ the target space $\widetilde{\Lambda}$ contains a subgroup isomorphic to $\Z$ generated by g. \end{proof}

\begin{Exm}\label{example1} Let $W^4$ be a Mazur manifold as in Figure~\ref{figure2}. There are various ways to see that the boundary is not the 3-sphere. Its fundamental group is known to be non-trivial \cite{lf1}. 

\begin{figure}[ht]
\begin{center}  
\includegraphics[scale=.8]{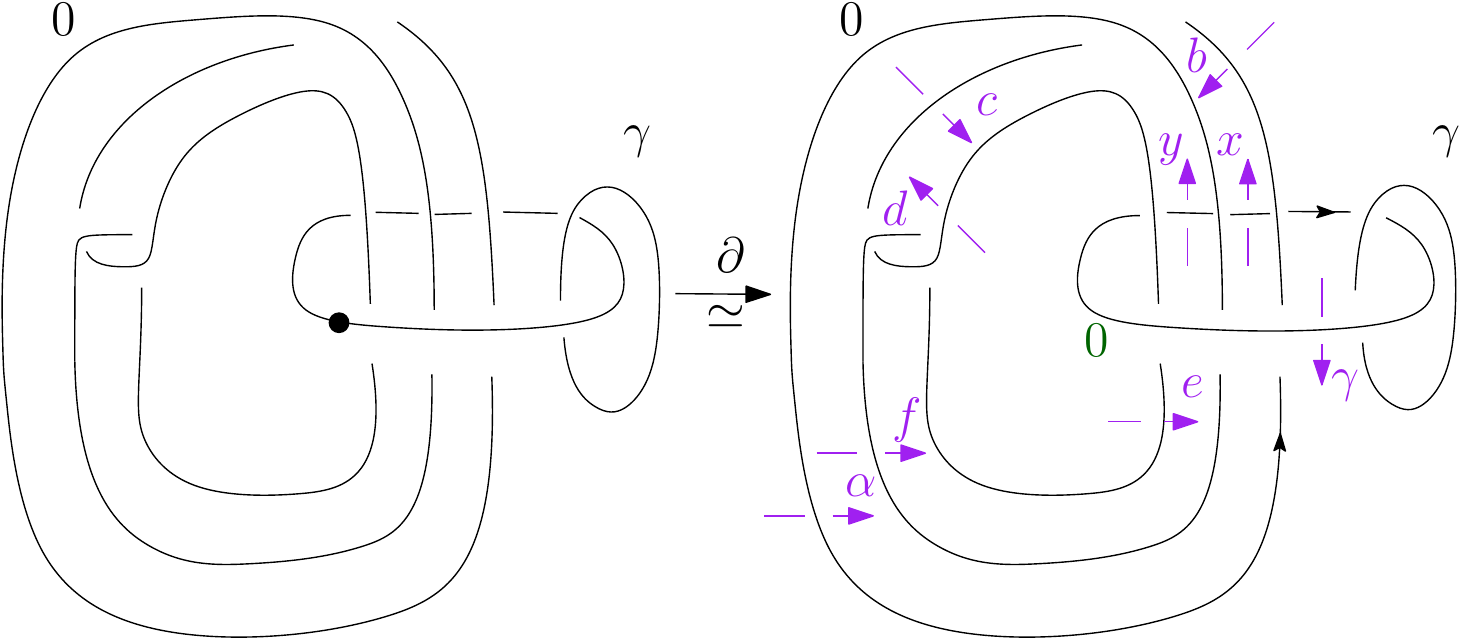}   
\caption{A homology sphere, Wirtinger presentation} 
\label{figure2} 
\end{center}
\end{figure}

\noindent By using Wirtinger presentation we describe the fundamental group:
$$\begin{array}{llrc}
\pi_1(Y)& = \left\lbrace\right.  \gamma, \alpha | & \gamma^2 \alpha \gamma^{-1} \alpha\gamma^{-1}\alpha^{-1} \gamma \alpha^{-1} \gamma \alpha^{-1} \gamma^{-1} \alpha \gamma^{-1} \alpha=1 &   \\ 
& & \gamma^{-1} \alpha \gamma^{-1} \alpha^2 \gamma \alpha \gamma^{-2} \alpha^{3} =1 & \left. \right\rbrace \end{array}.$$

\noindent Notice that setting $\gamma =1$ in this presentation would make this group trivial, hence $\gamma $ is a nontrivial element of $\pi_{1}(Y)$. To construct an example corresponding to Theorem~\ref{main1}, take an unknot and push an arc along a nontrivial loop $\gamma$ we get left diagram of Figure~\ref{figure3}. Obviously $\mu(k_1) = \gamma^{\pm} \in \widetilde{\Lambda}$ is nontrivial. Hence it is not almost-concordant to the unknot. On the other hand by iterating this process (i.e. increasing the number of twists) we can construct infinitely many null-homotopic knots $k_{n}$ with distinct $\mu$ invariant in the homology sphere, see the right diagram of Figure~\ref{figure3}.

\begin{figure}[ht]
\begin{center}  
\includegraphics[scale=.8]{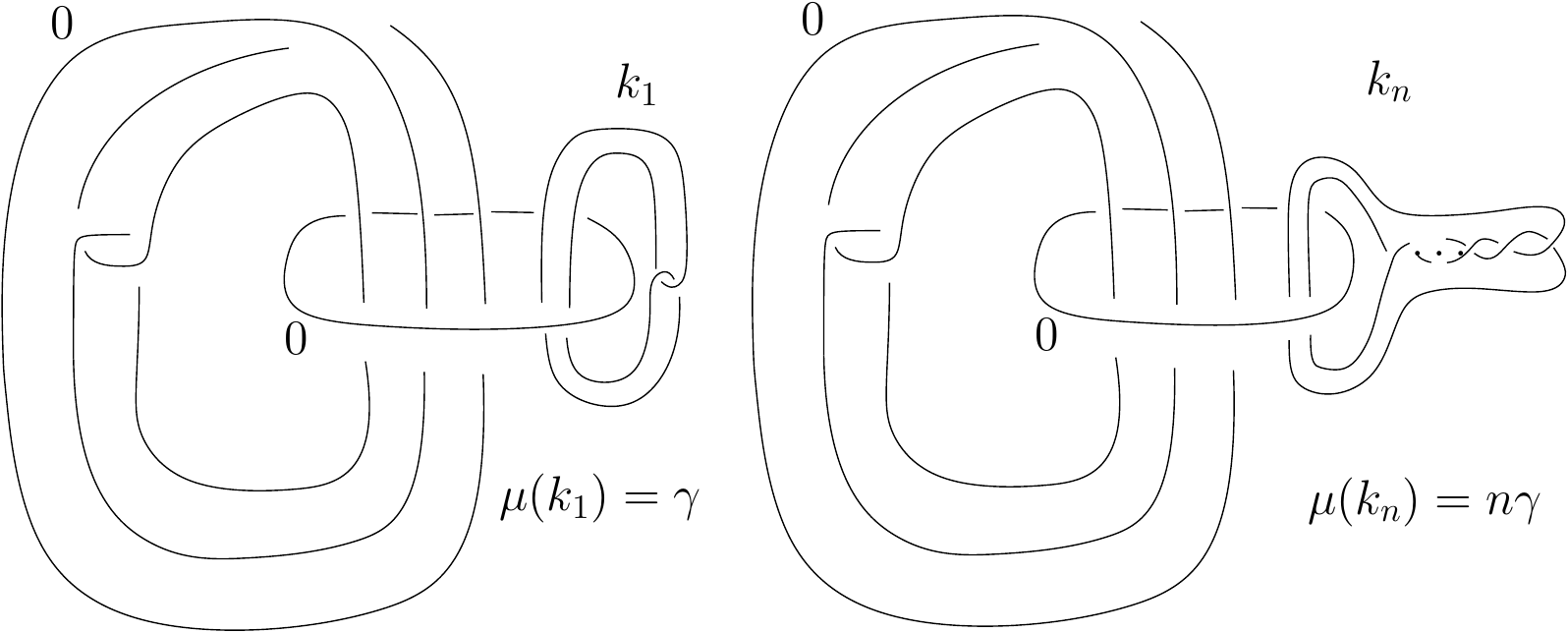}   
\caption{Distinct almost-concordant family} 
\label{figure3} 
\end{center}
\end{figure}

\end{Exm}

\section{Proof of Theorem 2}

\begin{proof}[Proof of Theorem \ref*{main2}]

First we introduce a (genus zero) cobordism move to a knot $k$, which starts with $k$, and ends with a two components link, consisting of the knot obtained from $k$ by changing one of its crossings union a small linking circle, as shown in Figure~\ref{figure7}. Let $k$ be a knot freely homotopic to $k'=S^1\times pt$ in $S^1\times S^2$, one can go from $k$ to $k'$ by finitely many crossing changes and isotopies. Change all the necessary crossings of $k$ by the cobordism described above.  Notice that for every crossing change, we get a small linking circle to the resulting knot. See Figure~\ref{figure8} as an example. It is obvious from Figure~\ref{figure9} that all those small circles which link $k'$ bound disks in $S^1\times S^2$ disjoint from $k'$. We accomplish this by sliding over the $0-$framed circle. By capping with disks these unknots we get a concordance from $k$ to $k'$ in $S^1\times S^2$. \end{proof}

 \begin{figure}[ht]
 \begin{center}  
 \includegraphics[scale=.8]{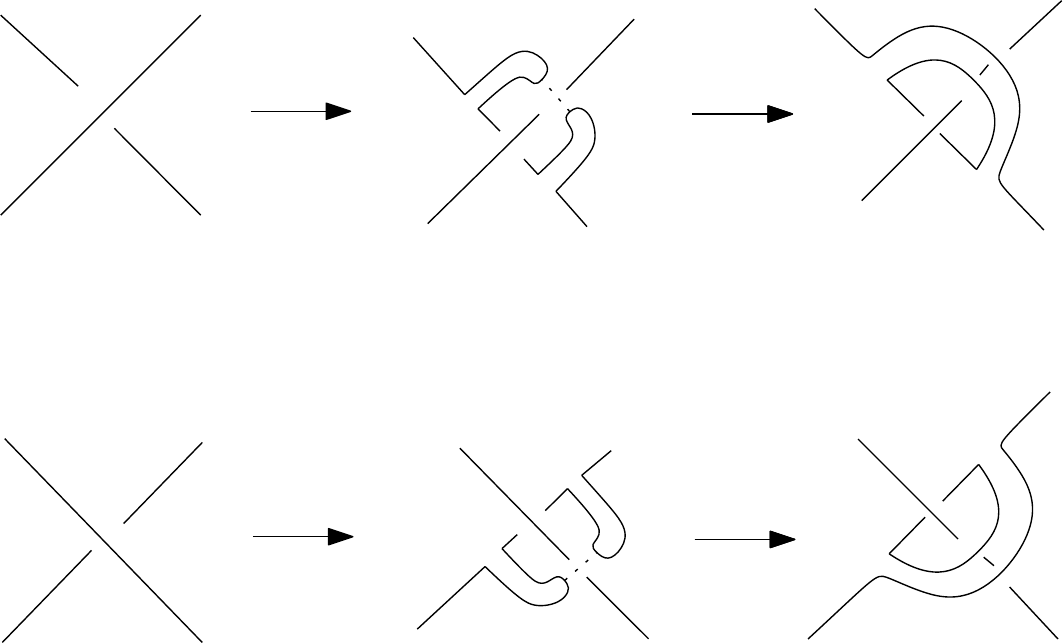}   
 \caption{Crossing change} 
 \label{figure7} 
 \end{center}
 \end{figure}

  \begin{figure}[ht]
  \begin{center}  
  \includegraphics[scale=.8]{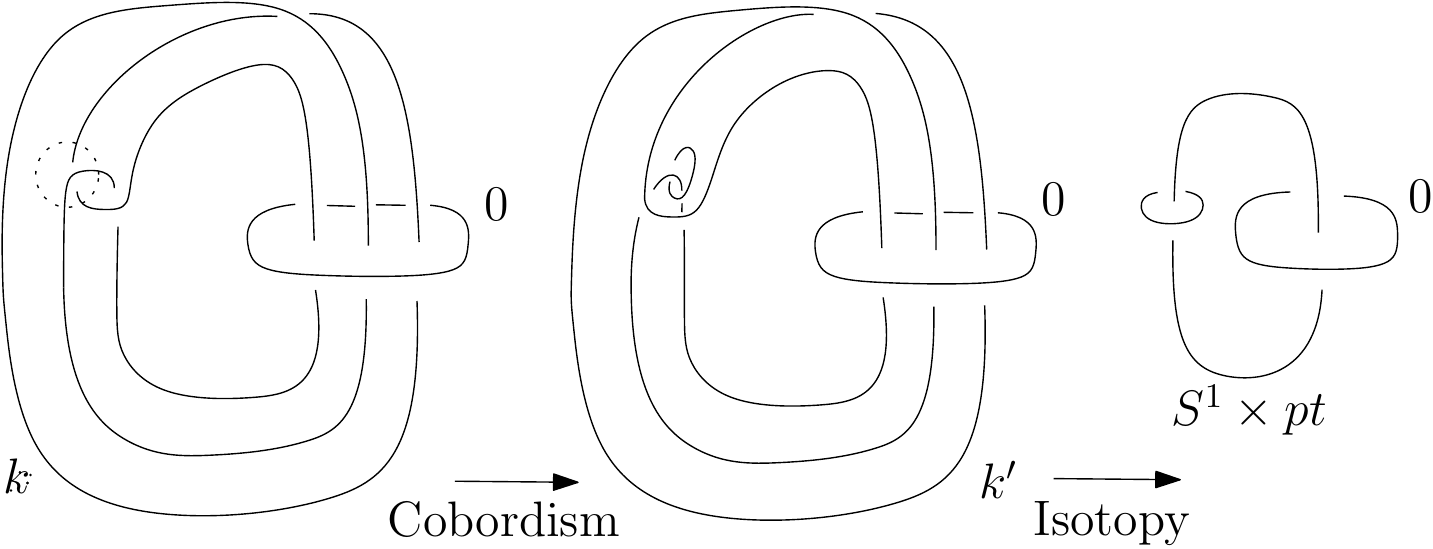}   
  \caption{An example of crossing change} 
  \label{figure8} 
  \end{center}
  \end{figure}

  \begin{figure}[ht]
 \begin{center}  
 \includegraphics[scale=.8]{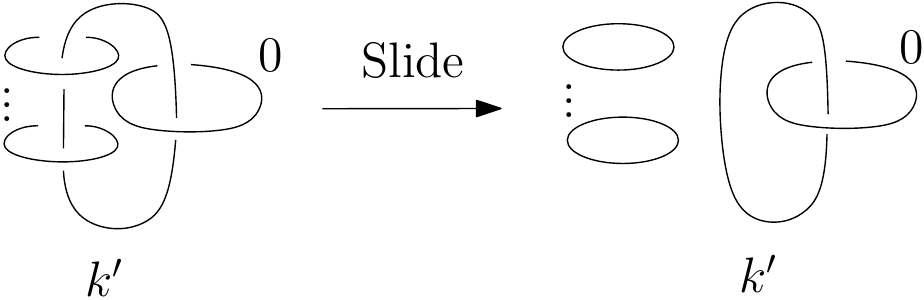}   
 \caption{Sliding and capping with disks} 
 \label{figure9} 
 \end{center} 
 \end{figure}

\section{PL-Slice}\label{pl-slice}
The notion of almost-concordance is same as the PL-concordance in $Y\times I$. Indeed, if $k_1$ and $k_2$ are PL-concordant then we may assume without loss of generality, the concordance has only one singular point which locally looks like a cone over a knot $k$. It is easy to see $k_1 \# -k$ is smoothly concordant to $k_2$ by removing a ball around the cone point and connecting two boundary components by removing neighbourhood of an arc lying on the concordance connecting $k_1$ to $k$. On the other hand if we have an almost concordance between $k_1$ and $k_2$ i.e. $k_1\#k'_1$ is concordant to $k_2\#k'_2$, then push the local knots inside the $4$-manifold and take the cone over the knots in some local ball to get a PL-concordance. Basically this tells us the family of knots we construct in Example ~\ref{example1} in particular in Figure~\ref{figure3} can not bound a PL-disk in the collar of the manifold but it can still bound in a $4$-manifold which $Y$ bounds. 

\noindent Next we see that none of these family members $\alpha_n$ in Figure~\ref{figure5} bounds a PL-disk in the Mazur manifold $W^4$. 

\begin{figure}[ht]
 \begin{center}  
 \includegraphics[scale=.8]{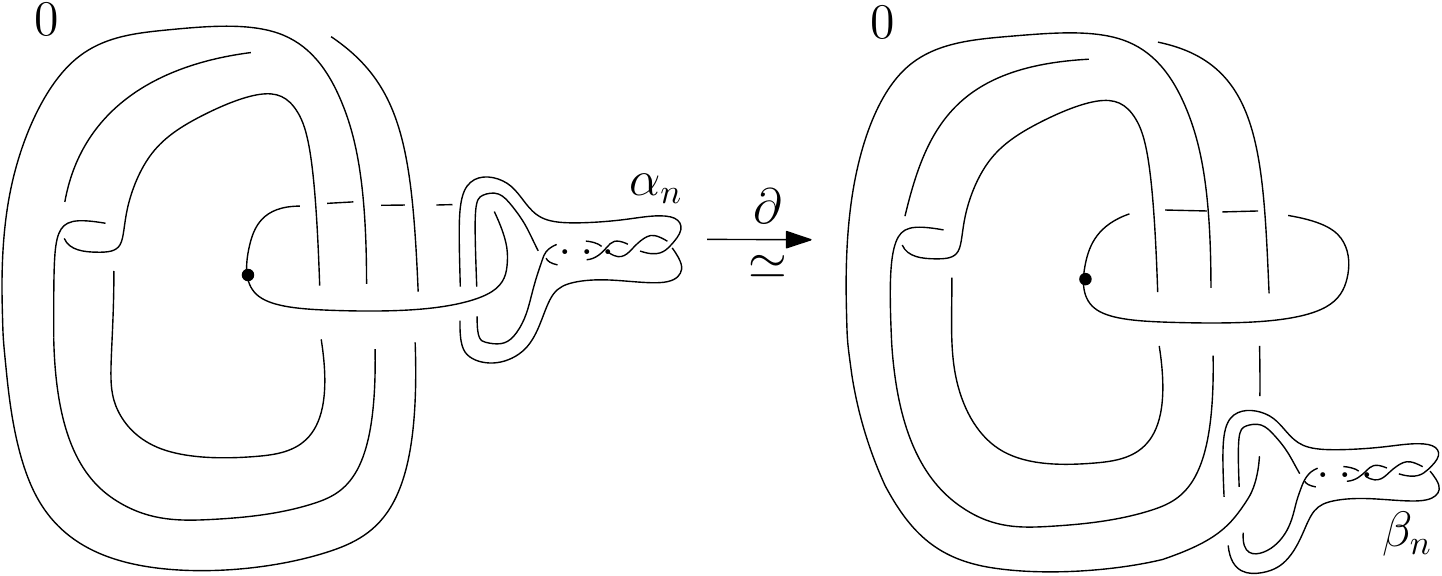}   
 \caption{Boundary diffeomorphism} 
\label{figure5}
 \end{center}
 \end{figure}

\noindent Here we imitate Akbulut's \cite{a1} \lq\lq A solution to a conjecture of Zeeman".  Observe that $W^4$ is a Stein domain by \cite{e1}. Consider the boundary diffeomorphism which takes $\alpha_n$ to $\beta_n$ as in Figure~\ref{figure5}, using $0 \leftrightarrow \bullet $ exchange and symmetry of the link surgery diagram of Mazur manifold. The knot $\beta_n$ is smoothly slice. To see that $\alpha_n$ is not slice we use the adjunction inequality as in \cite{am}. Let $F \subset W^4$ be a properly imbedded oriented surface in a Stein domain, such that $k = \partial F \subset \partial W^4$ is a Legendrian knot with respect to the induced contact structure. 

\noindent Let $f$ denote the framing of $k$ induced from the trivialization of the normal bundle of $F$, then:
  $$-\chi(F) \geq (tb(k) - f) + |rot(k)|.$$
\noindent Recall that the rotation number $rot(k)$, and the Thurston-Bennequin number $tb(k)$ are given by the formulae:
 $$rot(k) = \frac 1 2( \text{number of “downward” cusps} - \text{number of “upward” cusps} ),$$ 
  $$tb(k) = bb(k)- c(k),$$
 $bb(\alpha)$ is the blackboard framing (or writhe) of the front projection of $k$, and $c(k)$ is the number of right cusps.

\noindent Assume the curve $\alpha_n$ is slice so $\chi(F) = 1 $, $tb(\alpha_n) = 2n - (2n-1) = 1$, $rot(\alpha_n)=0$, $f=0 $, so we have a contradiction: $-1 \geq 1$, and therefore $\alpha_n$ is not slice. The same argument as in Theorem 1 \cite{a1} 
shows $\alpha_n$ does not bound a PL-disk in $W^4$.

\section{Topological Slice } Here we show that the family of knots that we constructed in the previous example are all topologically slice and therefore they are all distinct elements in the almost-concordance class of topologically slice knots on the boundary of the Mazur manifold.  

\noindent A knot $k$ in a homology sphere $Y$ has well-defined Alexander polynomial $\Delta_k(t) \in \Z\left[ t^{\pm} \right] $. Let $F$ be a Seifert surface of $k$ in $Y$ and $X$ be the knot complement. Then
$$\Delta_k(t) := \det (tS-S^T),$$  
where $S$ is an associated Seifert matrix of the bilinear form $\eta$
$$\begin{array}{cccc}
\eta: & \hg{1}(F;\Z) \times \hg{1}(F;\Z) & \to & \Z,\\
& \eta(\alpha,\beta)& = & lk(\alpha^+,\beta).
\end{array}$$
We adopt the convention $\alpha^+\in \hg{1}(X-F)$ is the image of $\alpha \in \hg{1}(F)$ via pushing $\alpha$ in the positive normal direction of $F$. As it is seen in Figure~\ref{figure4}, the Seifert surface $F$ of $k_n$ links the $0-$framed knot. One of its generators $x$ links that knot. In this case $lk(x^+,x)$ is not a direct calculation, since we have to find a Seifert surface $F_x$ ( or $F_{x^+}$) of $x$ (or $x^+$) to calculate  $lk(x^+,x)$. On the other hand, using the lemma below we can calculate the Seifert matrix easily.

\begin{Lem}[Lemma 7.13 of \cite{sa1}]
Let $k \cup l$ be a boundary link (i.e. knots $k$ and $l$ bound disjoint Seifert surfaces) in a homology sphere $Y$, and $Y'$ is a $\pm 1$ surgery of $Y$ along $k$. Then $\Delta_{l\subset Y}(t) = \Delta_{l'\subset Y'}(t)$. where $l' \subset Y'$ is the image of $l \subset Y$ under the surgery.
\end{Lem}

\noindent Since $\alpha$ and $k_n$ have disjoint Seifert surfaces, see the left diagram of Figure~\ref{figure4}, we perform $-1$ surgery on $\alpha$, and after some isotopy of $k_n$ we get the right diagram. Therefore for the Seifert matrix $S=\left( \begin{array}{cc}
0 & 1 \\
0 & n 
\end{array} \right) $ we have the corresponding Alexander polynomial
$$\Delta_{k_n \subset Y}(t) = \det \left( tS-S^T\right) =  t \dot{=} 1 .$$
Thanks to Freedman and Quinn's \cite{fq1} Theorem 11.7B these knots are all topologically slice.  

\begin{figure}[ht]
\begin{center}  
\includegraphics[scale=.8]{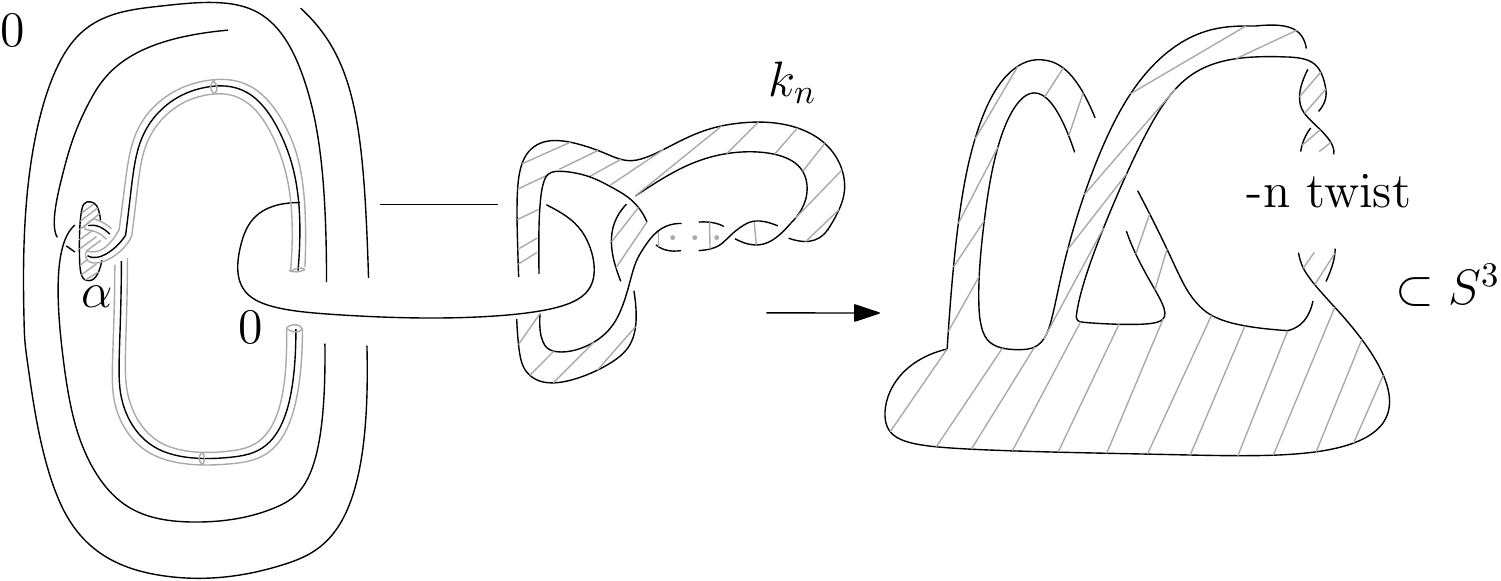}   
\caption{Alexander polynomial in homology sphere} 
\label{figure4} 
\end{center}
\end{figure}

\end{document}